\documentclass[a4paper,11pt]{article}
\usepackage[ngerman, english]{babel}
\usepackage[T1]{fontenc}
\usepackage[utf8]{inputenc}
\usepackage{amsmath}
\usepackage{amssymb}
\usepackage{amsthm}
\usepackage{graphicx}
\usepackage{here}
\usepackage{color}
\usepackage{xcolor}
\usepackage{enumerate}
\usepackage{lmodern}
\usepackage{fancyvrb}
\usepackage[plainpages=false]{hyperref}
\usepackage{caption}
\usepackage{subfigure}
\usepackage{epstopdf}
\usepackage{verbatim}
\newtheorem{defi}{Definition}[section]
\newtheorem{theo}[defi]{Theorem}
\newtheorem{prop}[defi]{Proposition}
\newtheorem{lem}[defi]{Lemma}
\newtheorem{con}[defi]{Conjecture}

\newtheorem{cor}[defi]{Corollary}

\newcounter{claimcount}
\usepackage{etoolbox}
\AtBeginEnvironment{proof}{\setcounter{claimcount}{0}}

\theoremstyle{remark}

\usepackage[colorinlistoftodos,prependcaption,textsize=scriptsize, textwidth=20mm,]{todonotes}

\title{Information dissemination and confusion in signed networks}

\author{Ligang Jin$^1$ Eckhard Steffen$^2$ \\
\footnotesize
$^1$ School of Mathematical Sciences, Zhejiang Normal University, Yingbin Road 688, 321004 Jinhua,
China\\
\footnotesize	$^2$ Department of Mathematics, Paderborn University, Warburger Str.\ 100, 33098 Paderborn,
		Germany
\\ \footnotesize ligang.jin@zjnu.cn, es@upb.de}

\date{}

\begin{document}

\maketitle

\begin{abstract}
We introduce a model of information dissemination in signed networks. It is a discrete-time process in which uninformed actors incrementally receive information from their informed neighbors or from the outside. Our goal is to minimize the number of confused actors - that is, the number of actors who receive contradictory information. 
We prove upper bounds for the number of confused actors in signed networks and 
in equivalence classes of signed networks.
In particular, we show that there are signed networks where, for any information placement 
strategy, almost 
60\% of the actors are confused. Furthermore, this is also the case when considering 
the minimum number of confused actors within an equivalence class of signed graphs.
\end{abstract}


\section{Introduction and basics}

In 1953, Harary \cite{Balance} introduced the notion of signed graphs
and laid the foundation for the study of signed graphs. 
As an early application of signed graphs,
Cartwright and Harary \cite{Cartwright_Harary_1956}  set Heider's theory \cite{Heider_1944} of structural balance in (social) networks into the graph theoretical concept of balance in signed graphs. 
From that on, signed graphs became a very active area of research.
Signed graphs are of significant interest from the mathematical point of view but there are many applications of signed graphs in 
many disciplines such as natural sciences and social sciences.
Thus, they are studied from very different perspectives.
The dynamic survey \cite{zaslavsky2012mathematical} gives an impression of the vast existing literature on signed graphs and related topics.

Signed graphs are natural concepts for designing and analyzing networks, see \cite{leskovec2010signed}. 
The actors of the network are represented by the vertices of a graph and 
edges represent a relation between (two) actors. Thus, the embedding of an actor into a network provides some information about their role within it.
Prominent measures of an actor's role in the network are centrality and community detection, see \cite{bloch2023centrality, liu2020simple, sun2020stable}.
These parameters provide insights into the network structure and they are used for 
interventions in (social) networks such as selecting actors 
for information placement \cite{intervensions}. Such research approaches 
are becoming increasingly relevant with the transition to cyber-physical social systems \cite{zhou2019cyber}, where technical and human actors are networked together. 
In these systems, the dissemination of information can
already be considered in their design \cite{dressler2018cyber}. 

We will introduce a model of information dissemination in signed networks. It is a discrete-time process where uninformed actors receive information from their informed neighbors or from the outside. As in social networks, an actor may receive contradictory information from members of its community. In this case, it is considered confused and does not actively spread information in the following. The process ends after a finite number of steps when every network actor is either informed or confused. In this paper we are interested in minimizing the number of confused actors. 

The paper is organized as follows. 
Next we give the necessary definitions and basic results on signed graphs.
In Section \ref{Sec: (G,sigma) special cases}, we introduce an information dissemination algorithm for signed graphs. We determine the minimum number of confused vertices for some signed graphs 
and prove some upper bounds for this parameter. We show that there are balanced 
signed graphs where, for any information placement strategy, almost 50\% 
of the vertices are confused. In general, there are 
signed graphs where, for any information placement strategy, almost 
%\(\frac{3}{5}\)
60\% of the vertices are confused. 
 
In Section \ref{Sec: Equivalence Classes}
we slightly relax the information dissemination algorithm which allows to study 
the minimum number of confused vertices within an equivalence class of signed graphs.
Here the situation changes for balanced signed graphs. We show that for any balanced
signed graph there exists a relaxed information placement strategy such that no vertex is confused. However, in general we stuck with a portion of almost 60\% of confused vertices. 
We then show that the minimum number of confused vertices in the equivalence class of $(G,\sigma)$
is equal to the minimum number of confused vertices in the equivalence class of $(G,-\sigma)$. 
We close with Section \ref{Sec: Problems}, where we state some problems and conjectures.

\subsection{Basics on graphs and signed graphs}
We consider finite simple graphs. The vertex set of a graph $G$ is denoted by $V(G)$
and its edge set by $E(G)$. Let $S \subseteq V(G)$. The subgraph of $G$ which is induced by $S$ is denoted by $G[S]$. 
Denote by $\partial_G(S)$ the set of edges having exactly one end in $S$.
The set of neighbors of vertices of $S$ in $V(G) \setminus S$
is denoted by $N_G(S)$. If $S = \{v\}$, then we write $N_G(v)$ instead of 
$N_G(\{v\})$. Let $s,t \geq 1$ be integers. The complete graph on 
$t$ vertices is denoted by $K_t$ and $K_{s,t}$ denotes the complete bipartite graph 
with partition sets $A, B$ with $\vert A \vert = s$ and $\vert B \vert = t$.

A \emph{signed graph} $(G,\sigma)$ is a graph $G$ together with a function 
$\sigma : E(G) \rightarrow \{ \pm  \}$, where $\{ \pm \}$ is seen as a multiplicative group.
The function $\sigma$ is called a \emph{signature} of $G$ and $\sigma(e)$ is called the 
\emph{sign} of an edge $e$. An edge $e$ is \emph{negative} if $\sigma(e) = -$ and
it is \emph{positive} otherwise. 
The set of negative edges is denoted by $E_{\sigma}^-$, and 
$E(G) \setminus E_{\sigma}^-$ is the set of positive edges. 
If $\sigma(e)=+$ for all $e \in E(G)$, then $\sigma$ is the \emph{all-positive} signature and 
it is denoted by $\texttt{\bf +}$,
and if $\sigma(e)=-$ for all $e \in E(G)$, then $\sigma$ is the \emph{all-negative} signature and it is denoted by $\texttt{\bf -}$.
The function $-\sigma$ defines another signature of $G$ which reverses all the signs of edges by $\sigma$.
The graph $G$ is sometimes called the \emph{underlying graph} of the signed graph $(G,\sigma)$.  

Let $(G',\sigma_{|_{E(G')}})$ be a subgraph of $(G,\sigma)$. 
The sign of $(G',\sigma_{|_{E(G')}})$ is the product of the signs
of its edges. 
A \emph{circuit} is a connected 2-regular graph. A circuit of length $n \geq 3$ will be denoted by $C_n$.
A circuit is \emph{positive} if its sign is $+$ and \emph{negative} otherwise. 
A subgraph $(G',\sigma_{|_{E(G')}})$ is \emph{balanced}
 if all circuits in $(G',\sigma_{|_{E(G')}})$ are positive, otherwise it is \emph{unbalanced}. Furthermore, 
positive (resp., negative) circuits are also often 
called balanced (resp., unbalanced) circuits.

A \emph{switching} of a signed graph $(G,\sigma)$ at $X \subseteq V(G)$ defines a signed graph $(G,\sigma')$ which is obtained from $(G,\sigma)$ by reversing the sign of each edge of the edge cut $\partial_G(X)$, 
i.e.~$\sigma'(e) = - \sigma(e)$ if $e \in \partial_G(X)$ and
$\sigma'(e) = \sigma(e)$ otherwise.
If $X= \{v\}$, then we also say that 
$(G,\sigma')$ is obtained from $(G,\sigma)$ by switching at $v$.
Switching defines an equivalence relation on the set of all signed graphs on $G$. 
We say that $(G,\sigma_1)$ and $(G,\sigma_2)$ are \emph{equivalent} if they can be obtained from each other by a switching at a vertex set $X$.  
We also say that $\sigma_1$ and $\sigma_2$ are equivalent signatures of $G$.
Harary \cite{Balance} proved the following characterization of balanced signed graphs.

\begin{theo} [\cite{Balance}] \label{Thm: character_balanced}
A signed graph $(G,\sigma)$ is balanced if and only if $V(G)$ can be partitioned into two sets $U_1$ and $U_2$ (possibly empty)
such that all edges of $E(G[U_1]) \cup E(G[U_2])$ are positive and all edges of $\partial_G(U_1)$ are negative.  
\end{theo}

By Theorem \ref{Thm: character_balanced} we have that $(G, \sigma)$ is balanced if and only if it is equivalent to $(G,\texttt{\bf +})$. 
A signed graph $(G,\sigma)$ is \emph{antibalanced} if and only if it is equivalent to 
$(G,\texttt{\bf -})$. 

It turns out that signed graphs are completely determined in terms of their
negative (positive) circuits. 

\begin{theo} [\cite{SignedGraphs}] \label{equ_classes}
Two signed graphs $(G,\sigma)$ and $(G,\sigma')$ are equivalent if and only if they have the same set of negative circuits. 
\end{theo}

The \emph{frustration index} of $(G,\sigma)$, denoted by $l(G,\sigma)$,
is the minimum cardinality of a set $E \subseteq E(G)$ such that
$(G-E, \sigma \vert_{E(G)-E})$ is balanced. The following result is well known, see e.g. Lemma 1.1 in \cite{cappello2022frustration}.

\begin{lem} \label{Lem: Min_signature}
Let $(G,\sigma)$ be a signed graph with $l(G,\sigma)=k$. 
If $E$ is a set of $k$ edges such that $(G-E, \sigma\vert_{E(G)-E})$ is balanced, 
then there is an equivalent signature $\gamma$ of $G$ such that $E_{\gamma}^- = E$.
\end{lem}

\section{Information dissemination on signed graphs} \label{Sec: (G,sigma) special cases}

We start with giving an informal description of the information dissemination algorithm.  

Let $(G,\sigma)$ be a signed graph. At step $0$ all vertices are uninformed. 

In step $i > 0$ first give information $A$ to an uninformed vertex $v_i$ and then, 
every informed vertex passes his information to its uninformed neighbors, where a negative edge
reverses the sign of the information. It could be that a vertex $z$ receives contradictory information from its neighbors (some send information $A$ and others send $-A$). Then $z$ is considered to be confused (it will get the status $C$ as being confused) and it does not send out any information 
in the following steps (which is equivalent to removing all edges between
$z$ and its uninformed neighbors). 

Eventually, every vertex of $(G,\sigma)$ has either information $A$  or $-A$ or it is confused. 
We want to minimize the number of confused vertices. 
The {\em Information Dissemination Problem} ($IDP$ for short) is the problem to get all vertices of
a signed network informed by such a process with minimum number of
confused vertices. Next we will give a formal definition of the problem. 

We start with defining an algorithm for information dissemination on signed graphs,
$ID$ for short. 

\begin{defi} [Algorithm for Information Dissemination ($ID$) on a Signed Graph $(G,\sigma)$]

Let $(G,\sigma)$ be a signed graph, with fixed signature $\sigma$. Define
step-wise functions 
$\gamma_i \colon V(G) \rightarrow \{A, -A, C, 0\}$ and $p$-vertices $v_i$.

$i=0$: Set $\gamma_0(v) = 0$ for every $v \in V(G)$ and $V_0[0] = V(G)$.

$i \geq 1$: Choose $v_i \in V_{i-1}[0]$. 
Let $\gamma'_{i-1}(v) =A$ if $v=v_{i}$, and $\gamma'_{i-1}(v) = \gamma_{i-1}(v)$ otherwise.

For $X \in \{A, -A, C, 0\}$, let $V'_{i-1}[X] = \{v \colon \gamma'_{i-1}(v) = X\}$ and for $v \in V(G)$, define \\[.2cm]

 $   \gamma_i(v) =  
 \begin{cases}
%        A & \text{ if }  v=v_i \\
	\gamma'_{i-1}(v) & \text{ if } v \in V'_{i-1}[A] \cup V'_{i-1}[-A] \cup V'_{i-1}[C],\\
	\sigma(vz)\gamma'_{i-1}(z) & \text{ if } v \in V'_{i-1}[0], z \in N_G(v) \cap (V'_{i-1}[A] \cup V'_{i-1}[-A]) \text{ and for all} \\
	                                               & ~y \in N_G(v) \cap (V'_{i-1}[A] \cup V'_{i-1}[-A]): \sigma(vy) \gamma'_{i-1}(y) = \sigma(vz)\gamma'_{i-1}(z),\\
	C & \text{ if } v \in V'_{i-1}[0] \text{ and there are } z_1, z_2 \in N_G(v) \cap (V'_{i-1}[A] \cup V'_{i-1}[-A]) \\
	                                               & \text{ with }
	\sigma(vz_1) \gamma'_{i-1}(z_1) \not = \sigma(vz_2)\gamma'_{i-1}(z_2),\\
	0 & \text{ otherwise.}
\end{cases} $ \\

Set $V_i[0] = \{v \colon \gamma_i(v) = 0 \}$. \\[.2cm]
Repeat the process if $V_i[0] \not = \emptyset$.
\end{defi}

In each repetition of the process, the cardinality of the set of vertices with value $0$ is reduced by at least 1. Thus, the process terminates after a finite number of repetitions, say $t$, when 
$\gamma_t(v) \not = 0$ for each $v \in V(G)$. 
For $i \in \{1, \dots, t\}$ and $X \in \{A,-A,C,0\}$,
let $V_i[X] = \{v \colon \gamma_i(v) = X\}$.

We say that vertex $v \in V(G) \setminus \{v_1, \dots,v_t\}$ receives 
information from its neighbors in step $i \in \{1, \dots,t\}$, 
if $0 = \gamma_{i-1}(v) \not = \gamma_i(v)$. 
The vertices of $V_t[C]$ 
are those which received contradictory information in the process,
we call them the $c${\em -vertices}. The vertices $v_1, \dots, v_t$ where we place information $A$ (one in each step) are called the placement vertices, $p${\em -vertices} for short. 

The function $\gamma_t$ is a {\em solution of the $ID$} on $(G,\sigma)$ 
with $p$-vertices $v_1, \dots, v_t$ and value $C_{(G,\sigma)}(\gamma_t) = \vert V_t[C]\vert$. Informally, 
the solution $\gamma_t$ of the $ID$ represents an information placement strategy on $(G,\sigma)$ and 
$C_{(G,\sigma)}(\gamma_t)$ is the number of $c$-vertices in the signed graph. 
Clearly, different choices of $p$-vertices may lead to different solutions and different sets of $c$-vertices. 
We want to minimize the number of $c$-vertices. For a signed graph $(G,\sigma)$ we define: 
$$C(G,\sigma) = \min \{C_{(G,\sigma)}(\gamma) \colon \gamma 
\text{ is a solution of the } ID \text{ on } (G,\sigma)\}.$$

$C(G,\sigma)$ is called the {\em confusion number} of $(G,\sigma)$.  
The {\em Information Dissemination Problem for a signed graph} $(G,\sigma)$ ($IDP$ for $(G,\sigma)$, for short) is to find a solution $\gamma$ of the $ID$ on $(G,\sigma)$ with $C_{(G,\sigma)}(\gamma) = C(G,\sigma)$. 
Note that we do not minimize the number of $p$-vertices.

\subsection*{The $IDP$ for some classes of signed graphs}

We start with proving an upper bound for the confusion number of balanced and of antibalanced signed graphs which both play an exceptional role in the theory of signed graphs. 

Let $n = 2k \geq 6$ be an even number and $(G_n,\sigma_n)$ be the signed graph which is
obtained by two copies $H, H'$ of an all-positive 
complete graph on $k$ with vertex sets $V(H) = \{u_1, \dots , u_k\}$, $V(H') = \{u_1', \dots , u_k'\}$
and negative edges $e_i = u_iu_i'$ for $i \in \{1, \dots, k\}$. 
Note that $G_n$ is a $k$-regular graph of order $n$.

\begin{prop}\label{prop_G_n}
For every even $n \geq 6$, $(G_n, \sigma_n)$ is balanced and 
$(G_n, -\sigma_n)$ is antibalanced. Furthermore,
$C(G_n,\sigma_n) = C(G_n,\texttt{\bf $-$}) = \frac{n}{2}-2$, and
$C(G_n,-\sigma_n) = C(G_n,\texttt{\bf +}) = 0.$
\end{prop}

\begin{proof} Let $n = 2k$. 
The negative edges of $(G_n,\sigma_n)$ form a perfect matching of
$G_n$. By construction and Theorem \ref{Thm: character_balanced}, $(G_n,\sigma_n)$ is balanced and therefore, it
is equivalent to $(G_n,\texttt{\bf +})$, Consequently, $(G_n,-\sigma_n)$
is equivalent to $(G_n,\texttt{\bf -})$ and hence, $(G_n,-\sigma_n)$ is antibalanced. 

Let $\pi \in \{\sigma, -\sigma, \texttt{\bf +},\texttt{\bf -}\}$
be a signature on $G_n$. 
Since for any two vertices $x,y$ there is an automorphism on $(G_n,\pi)$ 
which maps $x$ to $y$, we can choose an arbitrary vertex as $v_1$ (i.e., the first $p$-vertex), say $v_1 = u_1$. 

Then, in the first step, $\gamma(u_1)=A$, 
$\gamma_1(u_i) = \pi(u_1u_i) A$ for all $i \in \{2, \dots, k\}$, 
$\gamma_1(u_1') = \pi(u_1u_1')A$, and $\gamma_1(u_j')=0$ for each $j \in \{2, \dots, k\}$.

We have to choose $v_2$ from $\{u_2', \dots, u_k'\}$, w.l.o.g., say $v_2 = u_2'$ and set $\gamma_2(u_2') = A$.
Then, for every $j \in \{3, \dots, k\}$, vertex $u_j'$ has three informed 
neighbors, namely $u_1',u_2'$, and $u_j$.

If $\pi \in \{\sigma_n, \texttt{\bf -}\}$, then 
$\pi(u_1'u_j')\gamma_1(u_1') \not = \pi(u_2'u_j')\gamma_1(u_2')$
and therefore, $\gamma_2(u_j') = C$, and the process terminates.
Since we do not have any other options to choose $p$-vertices as above, it follows that $C(G_n,\pi) = k-2 = \frac{n}{2}-2$.

If $\pi\in \{-\sigma_n, \texttt{\bf +}\}$, then 
$\pi(u_1'u_j')\gamma_1(u_1') = \pi(u_2'u_j')\gamma_1(u_2')
= \pi(u_ju_j')\gamma_1(u_j)$  
and therefore, $\gamma_2(u_j') \neq C$. Consequently, $C(G_n,\pi) = 0$.  
\end{proof}

\begin{theo} \label{Theo: bound balanced graphs}
 Let $G$ be a connected graph of order $n \geq 4$. If $(G,\sigma)$ is balanced, 
 then $C(G,\sigma) \leq \frac{n}{2}-2$ and the bound is attained for every even $n$.    
\end{theo}

\begin{proof}
    By Theorem \ref{Thm: character_balanced}, $V(G)$ can be divided into two sets $U_1$ and $U_2$, such that edges between $U_1$ and $U_2$ are all negative and 
    all other edges of $G$ are positive. W.l.o.g., let $n_i = |U_i|$ and assume that $n_1 \geq n_2$. Thus, $n_2 - 2 \leq \frac{n}{2}-2$.

    If $n_2 = 0$, i.e., $\sigma$ is all positive, then $C(G,\sigma) = 0$ and we are done. 

    Let $n_2 > 0$ and $v$ be a vertex of $U_1$ which has a neighbor (say $u$) in $U_2$. 
    Let $v_1=v$ be the first $p$-vertex.  
    Then choose step-wise $p$-vertices $v_i$ as long as $V_{i-1}[0]\cap U_1 \neq \emptyset$ for $i\geq 2$. 
    Let $v_k$ be the last $p$-vertex which could be chosen from $U_1$. Thus, 
    $U_1\subseteq V_k[A]$ and therefore, $\vert V_k[A] \vert \geq n_1$. 

    If $\vert V_k[-A] \vert > 1$, then 
    $\vert V_k[A] \cup V_k[-A] \vert \geq n_1 +2$ and therefore,
    $C(G,\sigma) \leq n_2 - 2 \leq \frac{n}{2}-2$.

    If $\vert V_k[-A] \vert \leq 1$, then $U_2=\{u\}$ and consequently $C(G,\sigma) = 0$.
    
    By Proposition \ref{prop_G_n} the bound is attained for any even $n\geq 6$. For $n=4$, 
    the bound is attained by a balanced circuit of length $4$.  
\end{proof}

Next we consider trees (which are always balanced) and circuits. 

\begin{prop} \label{Prop: tree, circuit, vertex}
    Let $(G,\sigma)$ be a connected signed graph and $k \geq 3$ be an integer.
    \begin{enumerate}
        \item If $G$ is a tree, then $C(G,\sigma) = 0$.
        \item If $k \not = 5$ or $\sigma \not = \texttt{\bf -}$, then $C(C_k, \sigma) = 0$, and $C(C_5,\texttt{\bf -})=1$.
    \end{enumerate}
\end{prop}

\begin{proof}
    1. Let $G$ be a tree. Take an arbitrary vertex as the first $p$-vertex $v_1$.
    For each step $i\geq 2$, take a vertex from $V_{i-1}[0]\cap N_G(V_{i-1}[A]\cup V_{i-1}[-A])$ as $v_i$. 
    
    We will show by induction on $i$ that for each step $i\geq 1$, $V_{i}[A]\cup V_{i}[-A]$ induces a connected graph and $V_i[C]=\emptyset$, which completes the proof. It is trivial for $i=1$. For the induction step, by the induction hypothesis and the choice of $v_i$, $\{v_i\} \cup V_{i-1}[A]\cup V_{i-1}[-A]$ induces a connected graph. Since $G$ has no circuits, there is no vertex which has two neighbors in $\{v_i\} \cup V_{i-1}[A]\cup V_{i-1}[-A]$. It follows that $V_i[C]=\emptyset$, since $V_{i-1}[C]=\emptyset$ by the induction hypothesis.

    2. Let $G=[u_1\ldots u_k]$ be a circuit of length $k$.
    First assume that $(G,\sigma)=(C_5,\texttt{\bf -})$. Take $u_1$ as the first $p$-vertex and $u_3$ as the second one. After two steps, the $ID$ process terminates with one $c$-vertex, which is $u_4$. By the symmetry of $(C_5,\texttt{\bf -})$, we do not have any other options for the choice of $p$-vertices as above. Therefore, $C(C_5,\texttt{\bf -})=1$.

    Next assume that $(G,\sigma)\neq (C_5,\texttt{\bf -})$. 
    Let $k-r\equiv t \text{ (mod $3$)}$ with $0\leq r \leq 2$. If $r=0$, then take  $p$-vertices $v_i=u_{2i-1}$ for $i\in \{1,\ldots,t\}$. 
    The $ID$ process terminates after $t$ steps with no $c$-vertices.
    If $r=1$, then $u_{2t+1}$ is the only uninformed vertex after $t$ steps by taking the same $p$-vertices as above. Then take $v_{t+1}=u_{2t+1}$ and the process terminates with no $c$-vertices. It remains to assume that $r=2$. 
    If $k\neq 5$, then take $v_1=u_1$, $v_2=u_5$, and $v_i=u_{2i+1}$ for $i\in \{3,\ldots,t\}$.
    The process terminates after $t$ steps with no $c$-vertices.
    For $k=5$, since $(G,\sigma)\neq (C_5,\texttt{\bf -})$, $G$ must contain a positive path of length 3, w.l.o.g, say $u_1u_2u_3u_4$. Take $p$-vertices $v_1=u_1$ and $v_2=u_4$. The process terminates after two steps with no $c$-vertices.
\end{proof} 

\begin{prop}\label{prop_H_n}
Let $G$ be a connected graph of order $n \geq 5$ and maximum degree $\Delta \geq 3$. 
For any signed graph $(G,\sigma)$:    
\begin{enumerate}
    \item If $\Delta \geq n - 2$, then $C(G,\sigma) = 0$.
    \item If $\Delta < n - 2$,  then $C(G,\sigma) \leq n - 2 - \Delta$ and the bound is attained for 
    any $\Delta$.
    \item $C(G,\sigma)\leq (1-\frac{2}{\Delta})n$.
\end{enumerate}
\end{prop}

\begin{proof}
1. Let $v_1$ be the first $p$-vertex, where $v_1$ is a vertex of maximum degree. If $\Delta = n - 2$, then let $v_2$ be the second $p$-vertex, where $v_2$ is
the vertex of $V(G) \setminus (N_G(v_1) \cup \{v_1\})$.

2. Take a vertex of maximum degree as the first $p$-vertex and choose other $p$-vertices arbitrarily. This results in a solution of $ID$ on $(G,\sigma)$ either in one step with no $c$-vertices or in at least two steps with at most $n - 2 - \Delta$ $c$-vertices.
    
    Proposition \ref{prop_G_n} shows that the upper bound $n - 2 - \Delta$ for $C(G,\sigma)$ is achieved by the signed graph $(G_n,\sigma_n)$ for any given $\Delta \geq 3$.

3. Take an arbitrary vertex as the first $p$-vertex $v_1$.
For each step $i\geq 2$, let $W_{i-1}=\{u\colon u\in V_{i-1}[0] \text{~and~} \exists x\in V_{i-1}[A]\cup V_{i-1}[-A] \text{~with~} ux\in E(G)\}$,
let $W''_{i-1}=\{u\colon u\in V_{i-1}[0] \text{~and~}\exists x,y\in V_{i-1}[A]\cup V_{i-1}[-A] \text{~with~} \sigma(ux)\gamma_{i-1}(x)\neq \sigma(uy)\gamma_{i-1}(y)\}$,
and let $W'_{i-1}=W_{i-1}\setminus W''_{i-1}$.
If $W''_{i-1}\neq \emptyset$, then choose $v_i$ from $W''_{i-1}$. 
If $W_{i-1}''= \emptyset$ and $W'_{i-1}\neq \emptyset$, then choose $v_i$ from $W'_{i-1}$.
If $W_{i-1}=\emptyset$ and $V_{i-1}[0]\neq \emptyset$, then  choose $v_i$ from $V_{i-1}[0]$.
Let the $ID$ process on $(G,\sigma)$ terminates after $t$ steps.
It follows from the choice of $p$-vertices that 
each $c$-vertex has at least two neighbors in some component of $G[V_{t}[A]\cup V_t[-A]]$ which is not a tree.

Let $H_1$ consist of all the components of $G[V_{t}[A]\cup V_t[-A]]$ that is not a tree, and let $H_2 = G[V_{t}[C]]$.
For $i \in \{1,2\}$,
let $n_i = \vert V(H_i)\vert$ and $m_i = \vert E(H_i)\vert$.  Let $m^* = \vert \partial_G(H_1)\vert$. 
Since each component of $H_1$ is not a tree, $m_1\geq n_1$.
Since each  vertex of $H_2$ has at least two neighbors in $H_1$, $m^*\geq 2n_2$. Since the maximum degree of $H_1$ is at most $\Delta$, $n_1 \Delta  \geq 2m_1+m^*$.
We can conclude from these inequalities that $n_1 \Delta \geq 2n_1+2n_2$. Consequently, $n_2\leq (1-\frac{2}{\Delta})n$, since $n_1+n_2\leq n$.
\end{proof}

Let $s,t\geq 3$ be two integers. Let $(K_{t,t},\tau_t)$ be the signed graph 
where $E^-_{\tau_t}$ forms a perfect matching.
Denote by $\mathcal{G}_{s,t}$ the class of signed graphs on $st$ vertices $u_{i,j}$, 
$i \in \{0, \dots,s-1\}$ and $j \in \{0, \dots,t-1\}$, and
$st^2$ edges such that $U_i\cup U_{i+1}$ induces a $(K_{t,t},\tau_t)$ or a $(K_{t,t},-\tau_t)$, where $U_i=\{u_{i,0},u_{i,1},\ldots,u_{i,t-1}\}$ and the addition on index runs modulo $t$.
See Figures \ref{fig:figure-1} and \ref{fig:figure-2} for instance of $(K_{t,t},\tau_t)$ and $\mathcal{G}_{s,t}$, respectively.

\begin{figure}
    \centering
    \includegraphics[width=10cm]{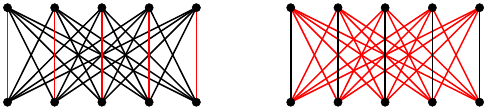}
    \caption{The signed graphs $(K_{5,5},\tau_5)$ (on the left-side) and $(K_{5,5},-\tau_5)$ (on the right-side), in which red lines (resp., black lines) represents negative edges (resp., positive edges).}
    \label{fig:figure-1}
\end{figure}

\begin{figure}
    \centering
    \includegraphics[width=6cm]{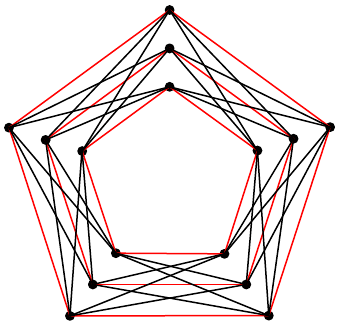}
    \caption{A member of the signed graph family $\mathcal{G}_{5,3}$}
    \label{fig:figure-2}
\end{figure}

\begin{prop}\label{prop_H_n}
Let $(H,\tau)$ be a signed graph of order $n$ and $t\geq 3$ be an integer.
\begin{enumerate}
    \item If $(H,\tau)\in\mathcal{G}_{6,t}$, then $C(H,\tau) =  \frac{n}{2}-4$.
    \item If $(H,\tau)\in\mathcal{G}_{5,t}$, then $C(H,\tau) =  \frac{3n}{5}-4$.
    \item If $(H,\tau)\in\mathcal{G}_{4,t}$, then $C(H,\tau) =  \frac{n}{2}-3$.
    \item If $(H,\tau)=(K_{t,t},\tau_t)$, then $C(H,\tau) = \frac{n}{2}-2$.
\end{enumerate}
\end{prop}

\begin{proof}
Let $(H,\tau)\in\mathcal{G}_{s,t}$.
By definition, for any $i$ and any two vertices $x,y \in U_i$, there is an automorphism on $(H,\tau)$ which maps $x$ to $y$.
Hence, we may assume that the first $p$-vertex is $v_1=u_{k0}$, with $k\in\{0,\dots,s-1\}$.  
Then in the first step, each vertex of $U_{k-1}\cup U_{k+1}\cup \{v_1\}$ receives precisely one of information $A$ and $-A$. 
Let $U'_i=U_i\setminus \{u_{i0}\}$ and $U=U'_{k-2}\cup U'_{k}\cup U'_{k+2}$.
Note that all the vertices of $U'_{k-1}$ receive the same information, and so do vertices of $U'_{k+1}$.
Hence, in the second step, no matter how the second $p$-vertex $v_2$ is chosen, all vertices of $U\setminus\{v_2\}$ receive contradictory information from $U'_{k-1}\cup U'_{k+1}$.
Hence, $C(H,\tau)\geq |U|-1.$

On the other hand, for $s\in\{4,5,6\}$, if take $v_2=u_{k+2,1}$, then in the second step, all the vertices of $(H,\tau)$ get informed and only the vertices of $U\setminus\{v_2\}$ become $c$-vertices. 
Therefore, $C(H,\tau)=|U|-1.$ For $s=4$, it gives $C(H,\tau)=2t-3=\frac{n}{2}-3$; for $s=5$, it gives $C(H,\tau)=3t-4=\frac{3n}{5}-4$; and for $s=6$, it gives $C(H,\tau)=3t-4=\frac{n}{2}-4.$ 

Let $(H,\tau)=(K_{t,t},\tau_t)$ with partition sets $A, B$, 
and $E^-_{\tau_t} = \{a_ib_i \colon i \in \{1, \cdots ,t\}, a_i \in A, b_i \in B \}$. 
W.l.o.g., choose $a_1$ as the first $p$-vertex, 
then $\gamma_1(b_1) = -A$ and $\gamma_i(b_i) = A$ if $i \geq 2$.
In the next step, we can only choose one vertex of $B$ as $p$-vertex and all other
vertices become $c$-vertices. Thus, $C(H,\tau) = \frac{n}{2}-2$.
\end{proof}

\section{Relaxed $IDP$ and equivalence classes of signed graphs} \label{Sec: Equivalence Classes}

The proof of Proposition \ref{prop_G_n} shows that
the balanced signed graphs $(G_n, \sigma_n)$ have this huge number of 
$c$-vertices since we are forced to give information $A$ 
second time. If we could give information $-A$, then we would not obtain any $c$-vertex. Similarly, if we would first switch to an all positive signed graph, then we also would not obtain any $c$-vertex. In this section we will
study these two variations of the $IDP$.

We first introduce a relaxed version of the $ID$ algorithm on a signed graph $(G,\sigma)$
($rID$ on $(G,\sigma)$ for short) where we
allow the placement of information $A$ or $-A$ on $p$-vertices. 
Second, we study bounds for the minimum number of confused vertices 
within an equivalence class of signed graphs. That is, for a signed graph
$(G,\sigma)$ we will study the parameter 
$$\min \{C(G,\pi) \colon (G,\sigma) \text{ and } (G,\pi) \text{ are switching equivalent} \}.$$ 

\begin{defi} [Algorithm for the relaxed Information Dissemination ($rID$) on a Signed Graph $(G,\sigma)$] 

Let $(G,\sigma)$ be a signed graph, with fixed signature $\sigma$. Define
step-wise functions 
$\gamma_i \colon V(G) \rightarrow \{A, -A, C, 0\}$ and $p$-vertices $v_i$ as follows:

$i=0$: Set $\gamma_0(v) = 0$ for every $v \in V(G)$ and $V_0[0] = V(G)$.

$i \geq 1$: Choose $v_i \in V_{i-1}[0]$ and $A'\in\{A,-A\}$. 
Let $\gamma'_{i-1}(v) =A'$ if $v=v_{i}$, and $\gamma'_{i-1}(v) = \gamma_{i-1}(v)$ otherwise.

For $X \in \{A, -A, C, 0\}$, let $V'_{i-1}[X] = \{v \colon \gamma'_{i-1}(v) = X\}$ and for $v \in V(G)$, let \\[.2cm]

 $   \gamma_i(v) =  
 \begin{cases}
	\gamma'_{i-1}(v) & \text{ if } v \in V'_{i-1}[A] \cup V'_{i-1}[-A] \cup V'_{i-1}[C],\\
	\sigma(vz)\gamma'_{i-1}(z) & \text{ if } v \in V'_{i-1}[0], ~z \in N_G(v) \cap (V'_{i-1}[A] \cup V'_{i-1}[-A]) \text{ and for all }\\ 
	                                               & ~~~~y \in N_G(v) \cap (V'_{i-1}[A] \cup V'_{i-1}[-A]): \sigma(vy) \gamma'_{i-1}(y) = \sigma(vz)\gamma'_{i-1}(z),\\
	C & \text{ if } v \in V'_{i-1}[0] \text{ and there are } z_1, z_2 \in N_G(v) \cap (V'_{i-1}[A] \cup V'_{i-1}[-A]) \\
	                                               & \text{ with }
	\sigma(vz_1) \gamma'_{i-1}(z_1) \not = \sigma(vz_2)\gamma'_{i-1}(z_2),\\
	0 & \text{ otherwise.}\\
\end{cases} $ \\

Set $V_i[0] = \{v \colon \gamma_i(v) = 0 \}$. \\[.2cm]
Repeat the process if $V_i[0] \not = \emptyset$.
\end{defi}

If the algorithm terminates after $t$ repetitions with function $\gamma_t$, then let 
$V_t[C] = \{v \colon \gamma_t(v) = C \}$.
The function $\gamma_t$ is a {\em solution of the $rID$} on $(G,\sigma)$ 
with $p$-vertices $v_1, \dots, v_t$ and value $C^r_{(G,\sigma)}(\gamma_t) = \vert V_t[C] \vert$. 

Informally, the solution $\gamma_t$ of the $rID$ represents an information placement strategy on $(G,\sigma)$ and 
$C_{(G,\sigma)}^r(\gamma_t)$ is the number of confused entities in the signed network. 
Clearly, different choices of $p$-vertices may lead to different solutions and different numbers of $c$-vertices. 
Again, the number of $c$-vertices should be minimized. For a signed graph $(G,\sigma)$ define 
$$C_r(G,\sigma) = \min \{C_{(G,\sigma)}^r(\gamma) \colon \gamma \text{ is a solution of the } rID \text{ on } (G,\sigma)\}.$$
$C_r(G,\sigma)$ is called the {\em relaxed confusion number} of $(G,\sigma)$.
The {\em relaxed Information Dissemination Problem for a signed graph} $(G,\sigma)$ ($rIDP$ on $(G,\sigma)$ for short) is to find a solution $\gamma$ of the $rID$ on $(G,\sigma)$
with $C_{(G,\sigma)}^r(\gamma) = C_r(G,\sigma)$. 

By definition, any solution of the $IDP$ on a signed graph $(G,\sigma)$ is a solution of the $rIDP$ on $(G,\sigma)$. Therefore, 
$C_r(G,\sigma)\leq C(G,\sigma)$.

\begin{lem} \label{Prop: same C_r} 
Let $(G,\sigma)$ and $(G,\pi)$ be signed graphs on $G$.
    If $(G,\sigma)$ and $(G,\pi)$ are equivalent, then $C_r(G,\sigma) = C_r(G,\pi)$.
\end{lem}

\begin{proof}
    Since $(G,\sigma)$ and $(G,\pi)$ are switching equivalent, we may assume that $\sigma$ can be obtained from $\pi$ by switching at a set $S$ of vertices.
    Take any solution $\gamma$ of $rIDP$ on $(G,\pi)$ with $p$-vertices $v_1,\ldots, v_t$. For $i\in\{1,\ldots, t\}$, let $\zeta_i(v_i)=-\gamma_i(v_i)$ if $v_i\in S$, and $\zeta_i(v_i)=\gamma_i(v_i)$ otherwise.
    We can see that each $v_i$ tells to any of its neighbors a same information under $\gamma$ as under $\zeta$.
    Hence, $\zeta$ is a solution of $rIDP$ on $(G,\sigma)$ with $C_r(G,\pi)$ many $c$-vertices, giving $C_r(G,\sigma)\leq C_r(G,\pi)$.
    Similarly, we can show that the inverse holds as well, i.e., $C_r(G,\pi) \leq C_r(G,\sigma)$.
    Therefore, two sides are equal, as desired.
\end{proof}

Next we show that solutions of the  $rIDP$ give the minimum value of an $ID$ solution on an equivalence class of signed graphs. 

\begin{theo} \label{Theo: relaxed confusion number}
    Let $(G,\sigma)$ be a signed graph. Then $$C_r(G,\sigma) = 
    \min \{C(G,\pi) \colon (G,\sigma) \text{ and } (G,\pi) \text{ are switching equivalent} \}.$$
\end{theo}

\begin{proof}
  Let $(G,\pi')$ be a signed graph with 
  $C(G,\pi') = \min \{C(G,\pi) \colon (G,\sigma)$ and  $(G,\pi)$ are switching equivalent$\}.$
  The theorem states that $C_r(G,\sigma)=C(G,\pi')$.
  Let $\gamma$ be an optimal solution (with minimum number of $c$-vertices) of the $rID$ on $(G,\pi')$  with $p$-vertices $v_1,\ldots,v_t$.
  Let $I=\{i\colon 1\leq i\leq t \text{ and } \gamma(v_i)=-A\}$, and denote by $(G,\tau)$ the signed graph obtained from $(G,\pi')$ by switching at $\{v_i\colon i\in I\}$. So, $\gamma$ is a solution of $ID$ on $(G,\tau)$ with the same $p$-vertices $v_1,\ldots,v_t$.
  It follows that $C(G,\tau)\leq C_r(G,\pi') \leq C(G,\pi')$.
  The minimality of $C(G,\pi')$ implies $C(G,\pi') \leq C(G,\tau)$ and
  therefore, $C(G,\tau)=C_r(G,\pi')=C(G,\pi')$.
  By Lemma \ref{Prop: same C_r}, $C_r(G,\sigma)=C_r(G,\pi')$ and
  consequently, $C_r(G,\sigma)=C(G,\pi')$, as desired.
\end{proof}

Surprisingly, there is no difference on the relaxed confusion number between two signed graphs with same underlying graph and opposite signatures, as shown by the following theorem.

\begin{theo} \label{Theo: C_r(G,s)=C_r(G,-s)}
 For every signed graph 
 $(G,\sigma) \colon C_r(G,\sigma) = C_r(G,-\sigma)$.
\end{theo}

\begin{proof}
Let $\gamma_t$ be a solution of the $rIDP$ on $(G,\sigma)$
with $p$-vertices $v_1, \dots , v_t$ and a sequence of
functions $\gamma_0, \dots ,\gamma_t$, where
$\gamma_0(v)=0$ for each vertex $v$.
For $j \in \{1, \dots,t\}$, let the $p$-vertex $v_j$ be of level $j-1$, and a vertex $v\in V(G)\setminus\{v_1,\ldots,v_t\}$ is of level $j$, if $j$ is the smallest number such that $\gamma_j(v) \not = 0$ 
and $\gamma_i(v) = 0$ for every $i<j$.
Denote by $S_j$ the set of vertices of level $j$, and by $T_j$ the set of vertices of $S_j$ that are not confused. Clearly, $S_i\subseteq N_G(T_{i-1})$.

Define $\gamma'_t$ on $(G,-\sigma)$ as follows: 
$$   \gamma'_t(v) =  
 \begin{cases}
	\gamma_t(v) & \text{ if } v \text{ has even level }\\
    - \gamma_t(v) & \text{ if } v \text{ has odd level}, \text{ where we assume that } C=-C.\\
\end{cases}
$$

We claim that $\gamma'_t$ is a solution of the $rIDP$ on $(G,-\sigma)$
with $p$-vertices $v_1, \dots, v_t$.
We will show (by induction on $i$) a stronger statement that for each step $i$, every vertex $v\in S_i$ receives information $\gamma'_t(v)$ from its neighbors ($v$ 
receives information $C$ means it receives both $A$ and $-A$). 
We distinguish two cases according to the value of $\gamma'_t(v)$.

Case 1: assume that $\gamma'_t(v)\in \{A,-A\}$.
Since $S_i\subseteq N_G(T_{i-1})$, $v$ has a neighbor in $T_{i-1}$. For any $x\in T_{i-1}\cap N(v)$, $x$ has information $\gamma'_t(x)$ by induction hypothesis. 
So, $v$ receives information $-\sigma(xv)\gamma'_t(x)$ from $x$.
Note that under the solution $\gamma_t$ on $(G,\sigma)$, $v$ also receives information from $x$, which gives $\gamma_t(v)=\sigma(xv)\gamma_t(x)$. Moreover, 
since the levels of $v$ and $x$ have different parities, it follows from the definition of $\gamma'_t$ that $\gamma'_t(v)\gamma_t(x)=-\gamma'_t(x)\gamma_t(v)$.
Now we can conclude from above that $\gamma'_t(v)=-\sigma(xv)\gamma'_t(x)$, i.e., 
$v$ receives information $\gamma'_t(v)$ from $x$ under $\gamma'$.

Case 2: assume that $\gamma'_t(v)=C$. In this case, $\gamma_t(v)=C$ by the definition of $\gamma'_t$.
So, $v$ has two neighbors $u_1,u_2\in T_{i-1}$, such that $u_1$ sends information $A$ to $v$ and $u_2$ sends $-A$ under $\gamma_t$.
Note that for any $u\in\{u_1,u_2\}$, $v$ receives information $\sigma(uv)\gamma_t(u)$ from $u$ under $\gamma_t$, while it receives $-\sigma(uv)\gamma'_t(u)$ under $\gamma'_t$. Also note that $u_1$ and $u_2$ are of the same level. It follows from the definition of $\gamma'_t$ that $\{-\sigma(uv)\gamma'_t(u)\colon u=u_1,u_2\}$ equals to either 
$\{\gamma_t(u)\sigma(uv)\colon u=u_1,u_2\}$ or $\{-\gamma_t(u)\sigma(uv)\colon u=u_1,u_2\}$, both of which equals to $\{A,-A\}$.
So, $v$ receives information $\gamma'_t(v)$ from its neighbors under $\gamma'_t$.

This completes the proof of the claim. 
By the definition of $\gamma'_t$, $\{v \colon \gamma_t(v) = C\} = \{v \colon \gamma'_t(v) = C\}$.
Thus, $C_r(G,-\sigma) \leq C_r(G,\sigma)$. 

The direction $C_r(G,\sigma) \leq C_r(G,-\sigma)$ follows 
analogously and therefore, $C_r(G,\sigma) = C_r(G,-\sigma)$.
\end{proof}

As a consequence of Theorems \ref{Theo: relaxed confusion number} 
and \ref{Theo: C_r(G,s)=C_r(G,-s)} we obtain the 
desired and natural result for balanced and for antibalanced signed graphs. 

\begin{cor}\label{Cor: bound balanced graphs}
    Let $(G,\sigma)$ be a signed graph. If $(G,\sigma)$ is balanced
    or antibalanced, then $C_r(G,\sigma) = 0$.
\end{cor}

Surprisingly, the bound $\lceil \frac{3n}{5}-4 \rceil$ is achieved for the relaxed confusion number for the family $\mathcal{G}_{5,t}$ of signed graphs.

\begin{prop}\label{prop_relaxed-H_n}
Let $(H,\tau)$ be a signed graph of order $n$.
\begin{enumerate}
    \item If $(H,\tau)\in\mathcal{G}_{6,t}$, then $C_r(H,\tau) =  \frac{n}{2}-4$.
    \item If $(H,\tau)\in\mathcal{G}_{5,t}$, then $C_r(H,\tau) =  \frac{3n}{5}-4$.
    \item If $(H,\tau)\in\mathcal{G}_{4,t}$, then $C_r(H,\tau) =  \frac{n}{2}-3$.
    \item If $(H,\tau)=(K_{t,t},\tau_t)$, then $C_r(H,\tau) =  \frac{n}{2}-2$. 
\end{enumerate}
\end{prop}

\begin{proof}
    Note that the difference between $ID$ and $rID$ is that $p$-vertices are assigned with information $A$ in $ID$ while with information $A$ or $-A$ in $rID$. Hence, following the same proof as for Proposition \ref{prop_H_n}, this proposition can be verified as well.
\end{proof}

Since $C_r(G,\sigma)\leq C(G,\sigma)$ holds for any signed graph $(G,\sigma)$, the following statements are basically immediate consequences of Propositions \ref{Theo: bound balanced graphs}, \ref{Prop: tree, circuit, vertex} and  \ref{prop_H_n}.

\begin{cor} \label{Cor: transfer strong - relaxed}
    Let $n \geq 1$ be an integer and $(G,\sigma)$ be a signed graph of order $n$.
    \begin{enumerate}
        \item If $G$ is a tree, then $C_r(G,\sigma) = 0$.
        \item If $G$ is a circuit, then $C_r(G,\sigma) = 0$.
        \item If $\ell (G,\sigma) \leq 1$, then $C_r(G,\sigma) = 0$.
        \item If $G$ is of maximum degree $\Delta \geq 3$, then $C_r(G,\sigma) \leq n - 2 - \Delta$ and the bound is attained for 
    any $\Delta$.
        \item If $G$ is of maximum degree $\Delta \geq 3$, then $C_r(G,\sigma)\leq (1-\frac{2}{\Delta})n$.
    \end{enumerate}
\end{cor}

By Proposition \ref{prop_relaxed-H_n}, the bound in Corollary \ref{Cor: transfer strong - relaxed}.4 is achieved by the signed graph $(K_{\Delta,\Delta},\tau_\Delta)$ for any $\Delta$.

Corollary \ref{Cor: transfer strong - relaxed}.3 is a first result on the relation of 
the confusion number and the frustration index of a signed graph, which gives information
on the signed graphs within one equivalence class.

\begin{theo} \label{Prop: 2nd bound with frustration}
There are a family of signed graphs $(H_n,\sigma_n)$ of order $n$ for each even $n\geq 8$ such that 
$C_r(H_n,\sigma_n) < \ell(H_n,\sigma_n)$ 
and 
$\lim_{n \rightarrow \infty} \frac{C_r(H_n,\sigma_n)}{\ell(H_n,\sigma_n)} = 1$.   
\end{theo}

\begin{proof}
Let $n=2t$ and let $(H_n,\sigma_n)=(K_{t,t},\tau_t)$ with partition sets $A, B$, 
and $E^-_{\tau_t} = \{a_ib_i \colon i \in \{1, \cdots ,t\}, a_i \in A, b_i \in B \}$, see Figure \ref{fig:figure-1} for $(H_{10},\sigma_{10})$.

First we show that $\ell(H_n,\sigma_n)=\frac{n}{2}$. 
Since the removal of $a_1b_1,\ldots, a_tb_t$ results in an all positive (thus, balanced) signed graph, $\ell(H_n,\sigma_n)\leq t$. 
On the other side, notice that $C_1,C_2,\ldots,C_t$ are pairwise edge-disjoint negative circuits, where $C_i=[a_ib_ia_{i+2}b_{i+1}]$. 
Recall that a signed graph is balanced if and only if it contains no negative circuits.
Thus, any set of edges whose removal from $(H_n,\sigma_n)$ results in a balanced signed graph must contain at least one edge from each $C_i$. It follows that $\ell(H_n,\sigma_n)\geq t$.
Now we can conclude from above that $\ell(H_n,\sigma_n)=t=\frac{n}{2}$.

By Proposition \ref{prop_relaxed-H_n}, $C_r(H_n,\sigma_n) =  \frac{n}{2}-2$.
So, we can derive that $C_r(H_n,\sigma_n) < \ell(H_n,\sigma_n)$ 
and 
$\lim_{n \rightarrow \infty} \frac{C_r(H_n,\sigma_n)}{\ell(H_n,\sigma_n)} = 1$.
\end{proof}

\section{Problems and conjectures} \label{Sec: Problems}

Based on the aforementioned results, 
we propose the following problems and conjectures on the (relaxed)
confusion number of signed graphs. 

\begin{con} \label{Conj: bound_strong} 
For every signed graph $(G,\sigma)$ on $n$ vertices, $C(G,\sigma) \leq \lceil \frac{3n}{5}-4 \rceil$.   
\end{con}

Determining the confusion number or the relaxed confusion number of a 
signed graph seem to be hard problems. 
Another hard problem is to determine the smallest number $k$ 
for which there is a solution $\gamma_k$ of the $IDP$ or of the $rIDP$
for a signed graph. 
In other words, the latter problems ask for the minimum number $k$ of 
 $p$-vertices in a solution of the $IDP$ or the $rIDP$ 
on $(G,\sigma)$, respectively.
These questions are related to determining the burning
number $b(G)$ of a graph $G$, which had been introduced in \cite{bonato2014burning}. If $\sigma = \textbf{+}$, then 
determining $k$ for the $IDP$ and if $(G,\sigma)$ is balanced or antibalanced, then determining $k$ for 
the $rIDP$ on $(G,\sigma)$ is equivalent to determining $b(G)$, where $k \in \{b(G)-1, b(G)\}$, due to slight differences between the definitions of graph burning and information dissemination.
In \cite{bonato2016burn} it is conjectured that $b(G) \leq \lceil \sqrt{n} \rceil$ if $G$ is of order $n$. An approximate confirmation of the conjecture was given recently in \cite{norin2022burning}. Determining the burning number is $NP$-complete even for specific classes of trees \cite{hiller2020burning}. 

Let $(G,\sigma)$ be a signed graph. 
If $C(G,\sigma) \leq \lceil \frac{3n}{5}-4 \rceil$, then 
$C_r(G,\sigma) \leq \lceil \frac{3n}{5}-4 \rceil$.
With view on Theorem \ref{Prop: 2nd bound with frustration} 
one may expect that $C_r(G,\sigma) \leq \ell(G,\sigma)$. 
We conjecture the following to be true. 

\begin{con} Let $G$ be a graph of order $n \geq 1$. 
For every signed graph $(G,\sigma)$: 
$C_r(G,\sigma) \leq \min \{\ell(G,\sigma), \lceil \frac{3n}{5}-4 \rceil \}$.
\end{con}

From the application point of view, it would be interesting to analyse real world cyber-physical social networks with regard to information dissemination processes. In particular, 
it might be of further interest whether there are real world signed 
cyber-physical social networks (or any other real world application of that
information dissemination concept) where such a huge portion of the network actors 
(as proved in Proposition \ref{prop_relaxed-H_n}.2) will be confused.

\bibliography{IDSN_Lit}{}

\begin{thebibliography}{10}

\bibitem{bloch2023centrality}
F.~Bloch, M.~O. Jackson, and P.~Tebaldi.
\newblock Centrality measures in networks.
\newblock {\em Social Choice and Welfare}, 61(2):413--453, 2023.

\bibitem{bonato2014burning}
A.~Bonato, J.~Janssen, and E.~Roshanbin.
\newblock Burning a graph as a model of social contagion.
\newblock In {\em Algorithms and Models for the Web Graph: 11th International
  Workshop, WAW 2014, Beijing, China, December 17-18, 2014, Proceedings 11},
  pages 13--22. Springer, 2014.

\bibitem{bonato2016burn}
A.~Bonato, J.~Janssen, and E.~Roshanbin.
\newblock How to burn a graph.
\newblock {\em Internet Mathematics}, 12(1-2):85--100, 2016.

\bibitem{cappello2022frustration}
C.~Cappello and E.~Steffen.
\newblock Frustration-critical signed graphs.
\newblock {\em Discrete Applied Mathematics}, 322:183--193, 2022.

\bibitem{Cartwright_Harary_1956}
D.~Cartwright and F.~Harary.
\newblock Structural balance: a generalization of {H}eider's theory.
\newblock {\em Psychological Review}, 63:277 -- 293, 1956.

\bibitem{dressler2018cyber}
F.~Dressler.
\newblock Cyber physical social systems: Towards deeply integrated hybridized
  systems.
\newblock In {\em 2018 International Conference on Computing, Networking and
  Communications (ICNC)}, pages 420--424. IEEE, 2018.

\bibitem{Balance}
F.~Harary.
\newblock On the notion of balance of a signed graph.
\newblock {\em Michigan Mathematical Journal}, 2:143--146, 1953--54.

\bibitem{Heider_1944}
F.~Heider.
\newblock Social perception and phenomenal causality.
\newblock {\em Psychological Review}, 51:358 -- 378, 1944.

\bibitem{hiller2020burning}
M.~Hiller, A.~M. Koster, and E.~Triesch.
\newblock On the burning number of p-caterpillars.
\newblock In {\em Graphs and Combinatorial Optimization: from Theory to
  Applications: CTW2020 Proceedings}, pages 145--156. Springer, 2020.

\bibitem{leskovec2010signed}
J.~Leskovec, D.~Huttenlocher, and J.~Kleinberg.
\newblock Signed networks in social media.
\newblock In {\em Proceedings of the SIGCHI conference on human factors in
  computing systems}, pages 1361--1370, 2010.

\bibitem{liu2020simple}
W.-C. Liu, L.-C. Huang, C.~W.-J. Liu, and F.~Jord{\'a}n.
\newblock A simple approach for quantifying node centrality in signed and
  directed social networks.
\newblock {\em Applied Network Science}, 5:1--26, 2020.

\bibitem{norin2022burning}
S.~Norin and J.~Turcotte.
\newblock The burning number conjecture holds asymptotically.
\newblock {\em Journal of Combinatorial Theory, Series B}, 168:208--235, 2024.

\bibitem{sun2020stable}
R.~Sun, C.~Chen, X.~Wang, Y.~Zhang, and X.~Wang.
\newblock Stable community detection in signed social networks.
\newblock {\em IEEE Transactions on Knowledge and Data Engineering},
  34(10):5051--5055, 2020.

\bibitem{intervensions}
T.~Valente.
\newblock Network interventions.
\newblock {\em Science (New York, N.Y.)}, 337:49--53, 07 2012.

\bibitem{SignedGraphs}
T.~Zaslavsky.
\newblock Signed graphs.
\newblock {\em Discrete Applied Mathematics}, 4(1):47--74, 1982.

\bibitem{zaslavsky2012mathematical}
T.~Zaslavsky.
\newblock A mathematical bibliography of signed and gain graphs and allied
  areas.
\newblock {\em The Electronic Journal of Combinatorics}, pages DS8--Dec, 2012.

\bibitem{zhou2019cyber}
Y.~Zhou, F.~R. Yu, J.~Chen, and Y.~Kuo.
\newblock Cyber-physical-social systems: A state-of-the-art survey, challenges
  and opportunities.
\newblock {\em IEEE Communications Surveys \& Tutorials}, 22(1):389--425, 2019.

\end{thebibliography}
\addcontentsline{toc}{section}{References}
\bibliographystyle{abbrv}

\end{document}